\documentclass[12pt]{amsart}

\makeatletter

\@addtoreset{equation}{section}
\makeatother
\usepackage[margin=3cm]{geometry}

\usepackage[dvipdfmx]{graphicx}
\usepackage{geometry,amsmath,amssymb,bbm,amscd,graphicx}
\usepackage{mathrsfs} 
\usepackage[T1]{fontenc}
\usepackage{ytableau}
 \usepackage{mathtools}
\usepackage{tikz}
\usepackage{xcolor}

\allowdisplaybreaks[2]
\theoremstyle{definition}
\newtheorem{Theorem}[equation]{Theorem}

\newtheorem*{conj*}{Conjecture}
\newtheorem*{theorem*}{Theorem}
 
\newtheorem{Corollary}[equation]{Corollary}

\newtheorem{Lemma}[equation]{Lemma}
\newtheorem{Example}[equation]{Example}
\begin{document}

\begin{center}
{\Large\textbf{Duality formula and its generalization for Schur multiple zeta functions}}
\vspace{5mm}\end{center}

\begin{center}
{\large Maki Nakasuji\footnote{Supported by Grant-in-Aid for Scientific Research (C) 18K03223.}
and Yasuo Ohno\footnote{Supported by Grant-in-Aid for Scientific Research (C) 19K03437.} }
\vspace{20mm}\end{center}

\begin{center}
{\textsc{abstract}}\vspace{5mm}\\
\end{center}
{\footnotesize In the study on multiple zeta values, the duality formula is one of the families of basic relations and plays an important role in the 
investigation of algebraic structure of the space spanned by all multiple zeta values 
along with the generalized duality formula (so called Ohno relation) obtained by the second author.
In this article, we will discuss them for the Schur multiple zeta values which are the values at positive integers of 
the Schur multiple zeta function introduced by the first author, O. Phukswan and Y. Yamasaki.}
 \vspace{2mm}\\

\vskip 1cm
\par\noindent

%
%
%
\section{Introduction}
The classical multiple zeta and zeta-star values are defined by
$$\zeta(k_1, \ldots, k_r)=\sum_{0<m_1<\cdots < m_r}\frac{1}{{m_1}^{k_1}\cdots m_r^{k_r}},
\quad
\zeta^{\star}(k_1, \ldots, k_r)=\sum_{0< m_1\leq \cdots \leq  m_r}\frac{1}{{m_1}^{k_1}\cdots m_r^{k_r}}
$$
for $k_1, \ldots, k_{r-1}\geq 1$, $k_r>1$.
The ${\mathbb Q}$-vector spaces ${\mathcal Z}$ spanned by multiple zeta values and spanned by multiple zeta-star values
are the same, and its algebraic structure has been strongly investigated during these three decades.
One of the most popular results in this area is as follows: \vspace{2mm}\\
For any integer $k\geq 1$, we denote by ${\mathcal Z}_k$ the ${\mathbb Q}$-subspace of ${\mathcal Z}$ spanned by 
multiple zeta values of weight $k$. Then
the dimension of ${\mathcal Z}_k$
is less than or equal to $d_k$, where $d_k$ satisfies $\displaystyle\frac{1}{1-t^2-t^3}=\sum_k d_k t^k$. \vspace{2mm}\\
This is obtained by Goncharov and Terasoma (\cite{Go}, \cite{Te}, \cite{DG}). Further, there is a large difference in the number of convergence indices of 
multiple zeta values for each weight.
These results show that there are plenty of linear relations among these values.
However, their exact structure remains quite mysterious. 

Some families of relations which are important to understand the structure of the ring of multiple zeta values
${\mathcal Z}={\mathbb Q}+\sum_k {\mathcal Z}_k$
 are already known. Among them, the duality formula is one of the most important relations. This shows that
 all multiple zeta values except the selfdual case have a pair of multiple zeta values of the same weight with same values.
 The generalized duality formula called Ohno relation is a big family which includes 
 not only the duality but also other basic relations such as the sum formula and the derivation relation.
Moreover, by this formula we can see how  
the duality in the space of a small weight contributes the relation in the space of a big weight, which may help us
to understand the structure of them.
On the other hand, compared with that the duality for multiple zeta values is well known as above, the duality formula for multiple zeta-star values 
had not been obtained before introducing the Schur multiple zeta values.

The Schur multiple zeta functions introduced in \cite{NPY} are defined as sums over combinatorial objects called semi-standard Young tableaux
and generalize the classical multiple zeta and zeta-star functions.
Nakasuji-Phukswan-Yamasaki (\cite{NPY}) showed some determinant formulas for them such that Jacobi-Trudi, Giambelli and dual Cauchy formulas,
which lead to quite non-trivial algebraic relations among multiple zeta and zeta-star functions.
The Schur multiple zeta values are the values at the positive integers of  the Schur multiple zeta function,
which also generalize the classical multiple zeta and zeta-star values.
Nakasuji-Phukswan-Yamasaki (\cite{NPY}) established iterated integral representations of the Schur multiple zeta values of ribbon type, which yield a duality formula for multiple zeta-star values.

The primary goal of this article is to obtain the duality formula for the Schur multiple zeta values of skew type including ribbon type. And as the second,  
applying this formula, we will show the generalized duality formula called Ohno relation for the Schur type.

This article is organized as follows.
In Section 2, we review the basic terminology and some known results for the Schur multiple zeta function. 
In Section 3, we first give the definition of dual index (tableau) of convergence index for Schur multiple zeta value, and 
we will discuss the duality formula for the Schur multiple zeta values of skew type under some assumption.
In the proof, we will use the Jacobi-Trudi formula for the Schur multiple zeta function which is obtained in \cite{NPY} and the duality formula of the classical multiple zeta functions.
In Section 4,  we will consider the Ohno relation for the Schur type.
In this part, we will use the key lemma in the extended Jacobi-Trudi formula for the Schur multiple zeta functions which is shown in \cite{NT}
and the Ohno relation for the classical one.

\section{Preliminaries}
We review the Schur multiple zeta function from \cite{NPY}.
Let $\mathbb{N}$, $\mathbb{C}$ be the set of positive integers, and of complex numbers,
respectively.
Let $\lambda=(\lambda_1, \cdots, \lambda_m)$ be a non-increasing sequence of 
$n\in\mathbb{N}$, i.e. $\lambda_1\geq \lambda_2\geq \cdots \lambda_m>0$ with $|\lambda|:=\sum_i\lambda_i=n$.
Then a {\it Young diagram} $D_{\lambda}$ of shape $ \lambda$ is obtained by drawing $\lambda_i$ boxes in the $i$-th row.
We often identify $\lambda$ with the corresponding Young diagram.
Let $T_{\lambda}(X)$ be the set of all Young tableaux of shape $\lambda$ over a set $X$ and, in particular, $\mathrm{SSYT}_{\lambda}\subset T_{\lambda}(\mathbb{N})$ the set of all semi-standard Young tableaux of shape $\lambda$. Recall that $M=(m_{ij})\in \mathrm{SSYT}_{\lambda}$ if and only if $m_{i1}\le m_{i2}\le \cdots$ for all $i$ and $m_{1j}<m_{2j}<\cdots $ for all $j$. For  ${\pmb s}=(s_{ij})\in T_{\lambda}(\mathbb{C}),$ the Schur multiple zeta-function associated with $\lambda$ is defined as in \cite{NPY} by the series 
$$
\zeta_{\lambda}({ \pmb s})=\sum_{M\in \mathrm{SSYT}_{\lambda}}
{M^{ -\pmb s}}, 
$$
where $M^{ -\pmb s}=\displaystyle{\prod_{(i, j)\in \lambda}m_{ij}^{-s_{ij}}}$ for $M=(m_{ij})\in \mathrm{SSYT}_{\lambda}$. This series converges absolutely if ${\pmb s}\in W_{\lambda}$ where 
\[
  W_\lambda = W_\lambda({\mathbb C}) = 
\left\{{\pmb s}=(s_{ij})\in T_{\lambda}(\mathbb{C})\,\left|\,
\begin{array}{l}
 \text{$\Re(s_{ij})\ge 1$ for all $(i,j)\in D_{\lambda} \setminus C_{\lambda}$ } \\[3pt]
 \text{$\Re(s_{ij})>1$ for all $(i,j)\in C_{\lambda}$}
\end{array}
\right.
\right\}
\]
 with $C_{\lambda}$ being the set of all corners of $\lambda$.
 The Schur multiple zeta functions have a natural extension to those of skew type. Let $\lambda$ and $\mu$
 be partitions satisfying $\mu\subset \lambda$. The Young diagram of skew type $\lambda/\mu$ is the array of boxes contained in 
 $\lambda$ but not in $\mu$.
We may write $\delta=\lambda/\mu$.
The notations $T_{\delta}(X)$ for a set $X$, and $SSYT_{\delta}$ are the set of all Young tableau 
of skew type
and semi-standard Young tableaux of that type, respectively.
Then,
for  ${\pmb s}=(s_{ij})\in T_{\delta}(\mathbb{C}),$ the skew type Schur multiple zeta-function associated with $\delta$
is defined as in \cite{NPY} by the series 
$$
\zeta_{\delta}({ \pmb s})=\sum_{M\in \mathrm{SSYT}_{\delta}}
{M^{ -\pmb s}}, 
$$
where $M^{ -\pmb s}=\displaystyle{\prod_{(i, j)\in \delta}m_{ij}^{-s_{ij}}}$ for 
$M=(m_{ij})\in \mathrm{SSYT}_{\delta}$. This series converges absolutely if ${\pmb s}\in W_{\delta}$ where
$W_{\delta}$ is also similarly defined as $W_{\lambda}$.

In \cite{NPY}, they obtained some determinant formulas such as the Jacobi-Trudi, Giambelli and dual Cauchy formula for Schur
multiple zeta functions under certain assumption on variables. 
Nakasuji and Takeda \cite{NT} showed certain extended Jacobi-Trudi formula.
We here  review the (extended) Jacobi-Trudi formula with the key lemma.

A skew Young diagram is called a {\it ribbon} if it is connected and contains no $2\times 2$ block of boxes.
The maximal outer ribbon of $\lambda$ is called the rim of $\lambda$.
We peel the diagram $\lambda$ off into successive rims $\theta_t, \theta_{t-1}, \ldots, \theta_1$ beginning from the outside of $\lambda$
then a sequence $\Theta=(\theta_1, \ldots, \theta_t)$ of ribbons is called a {\it rim decomposition of $\lambda$}.
For $\lambda'=(\lambda_1', \ldots, \lambda_s')$, 
if each $\theta_i$ starts from $(1, i)$ for all $1\leq i\leq s$, then a rim decomposition $\Theta$ of $\lambda$ is called an {\it E-rim decomposition}.
 Here, we permit $\theta_i=\emptyset$. We denote by $Rim_E^{\lambda}$ the set of all $E$-rim decomposition of $\lambda$.
 
 \begin{Example}
 The following $\Theta=(\theta_1, \theta_2, \theta_3, \theta_4)$ is an $E$-rim decomposition of $\lambda=(4,3,3,2)$;
 $$
 \Theta=
 \begin{ytableau}
 1 & 2 & 3 & 4 \\
1 & 2 & 3\\
1 & 3 & 3\\
3 & 3
\end{ytableau},
 $$
 which means that 
 $\theta_1=
 \begin{ytableau}
{}\\
{}\\
{}
\end{ytableau}
$,
 $\theta_2=
 \begin{ytableau}
{}\\
{}
\end{ytableau}
 $, 
 $\theta_3=
 \begin{ytableau}
\none & \none & {}\\
\none & \none & {}\\
\none & {} &{} \\
{} & {}
\end{ytableau}
 $
and  $\theta_4=
 \begin{ytableau}
\\
\end{ytableau}
 $.

 \end{Example}
 
 Now consider the patterns  corresponding $E$-rim on the ${\mathbb Z}^2$ lattice.
 Fix $N\in {\mathbb N}$. For a partition $\lambda'=(\lambda'_1, \ldots, \lambda'_s)$, let $c_i$ and $d_i$ be lattice points in ${\mathbb Z}^2$
 respsctrively given by $c_i=(s+1-i, 1)$ and $d_i=(s+1-i+\lambda_i', N+1)$ for $1\leq i \leq s$.
 An {\it E-pattern} corresponding to $\lambda$ is a tuple
 $L=(\ell_1, \ldots, \ell_s)$ of directed paths on ${\mathbb Z}^2$, whose directions are allowed only to go one to the northeast or one up, such that $\ell_i$
 starts from $c_i$ and ends to $d_{\sigma_i}$ for some $\sigma\in {\frak S}_s$.
 Such $\sigma \in {\frak S}_s$ is called {\it type} of $L$ and we denote it by $\sigma={\rm{type}}(L)$.
 Let ${\mathcal E}_{\lambda}^N$ be the set of all $E$-patterns corresponding to $\lambda$. 
 
 \begin{Example}
Let $\lambda=(4,3,3,2)$. Then $\lambda'=(4,4,3,1)$ and
the following $L=(\ell_1,\ell_2, \ell_3, \ell_4) $ is one of the element in $\mathcal{E}^{6}_{(4,3,3,2)}$ for $\sigma=(1\;2\; 3)\in {\frak S}_4$.
\begin{figure}[h]
\begin{center}
 \begin{tikzpicture} 
  \node at (0.5,1) {$1$};
  \node at (0.5,2) {$2$};
  \node at (0.5,3) {$3$};
  \node at (0.5,4) {$4$};
  \node at (0.5,5) {$5$};
  \node at (0.5,6) {$6$};
  \node at (0.5,7) {$7$};
  \node at (1.7,5.5) {$\ell_4$};
    \node at (4.2,5.5) {$\ell_2$};
    \node at (5.2,5.5) {$\ell_1$};
    \node at (6.2,5.5) {$\ell_3$};
     \node at (1,0) {$c_4$};  
     \node at (2,0) {$c_3$};
     \node at (3,0) {$c_2$};
     \node at (4,0) {$c_1$};
     \node at (1,0.5) {$(1,1)$};  
     \node at (2,0.5) {$(2,1)$};
     \node at (3,0.5) {$(3,1)$};
     \node at (4,0.5) {$(4,1)$};
       \node at (2,7.5) {$(2,7)$};
       \node at (5,7.5) {$(5,7)$};
       \node at (7,7.5) {$(7,7)$}; 
       \node at (8,7.5) {$(8,7)$};
       \node at (2,8) {$d_4$};
       \node at (5,8) {$d_3$};
       \node at (7,8) {$d_2$}; 
       \node at (8,8) {$d_1$}; 
   \node at (1,1) {$\bullet$};
   \node at (1,2) {$\bullet$};
   \node at (1,3) {$\bullet$};
   \node at (1,4) {$\bullet$};   
   \node at (1,5) {$\bullet$};
   \node at (1,6) {$\bullet$};
   \node at (1,7) {$\bullet$};    
   \node at (2,1) {$\bullet$};
   \node at (2,2) {$\bullet$};
   \node at (2,3) {$\bullet$};  
   \node at (2,4) {$\bullet$};
   \node at (2,5) {$\bullet$};
   \node at (2,6) {$\bullet$};
   \node at (2,7) {$\bullet$};    
   \node at (3,1) {$\bullet$};
   \node at (3,2) {$\bullet$};
   \node at (3,3) {$\bullet$};  
   \node at (3,4) {$\bullet$};
   \node at (3,5) {$\bullet$};
   \node at (3,6) {$\bullet$};
   \node at (3,7) {$\bullet$};    
   \node at (4,1) {$\bullet$};
   \node at (4,2) {$\bullet$};
   \node at (4,3) {$\bullet$};  
   \node at (4,4) {$\bullet$};  
   \node at (4,5) {$\bullet$};
   \node at (4,6) {$\bullet$};
   \node at (4,7) {$\bullet$};    
   \node at (5,1) {$\bullet$};
   \node at (5,2) {$\bullet$};
   \node at (5,3) {$\bullet$};  
   \node at (5,4) {$\bullet$};  
   \node at (5,5) {$\bullet$};
   \node at (5,6) {$\bullet$};
   \node at (5,7) {$\bullet$};    
   \node at (6,1) {$\bullet$};
   \node at (6,2) {$\bullet$};
   \node at (6,3) {$\bullet$}; 
   \node at (6,4) {$\bullet$};  
   \node at (6,5) {$\bullet$};
   \node at (6,6) {$\bullet$};
   \node at (6,7) {$\bullet$};    
   \node at (7,1) {$\bullet$};
   \node at (7,2) {$\bullet$};
   \node at (7,3) {$\bullet$}; 
   \node at (7,4) {$\bullet$};  
   \node at (7,5) {$\bullet$};
   \node at (7,6) {$\bullet$};
   \node at (7,7) {$\bullet$};    
   \node at (8,1) {$\bullet$};
   \node at (8,2) {$\bullet$};
   \node at (8,3) {$\bullet$};  
   \node at (8,4) {$\bullet$};  
   \node at (8,5) {$\bullet$};
   \node at (8,6) {$\bullet$};
   \node at (8,7) {$\bullet$};    
 \draw (4,1) -- (5,2) -- (5,3) -- (5,4) -- (5,5) -- (6,6) -- (7,7);
 \draw[dotted] (3,1) -- (3,2) -- (3,3) -- (4,4) -- (4,5) -- (5,6) -- (5,7);
 \draw[loosely dashdotted] (2,1) -- (3,2) -- (4,3) -- (5,4) -- (6,5) -- (7,6) -- (8,7);      
 \draw[dashed] (1,1) -- (1,2) -- (1,3) -- (2,4) -- (2,5) -- (2,6) -- (2,7);
 \end{tikzpicture}.
\end{center}
\ \\[-10pt]
\caption{$L=(\ell_1, \ell_2, \ell_3, \ell_4)\in \mathcal{E}^{6}_{(4,3,3,2)}$}
\end{figure}
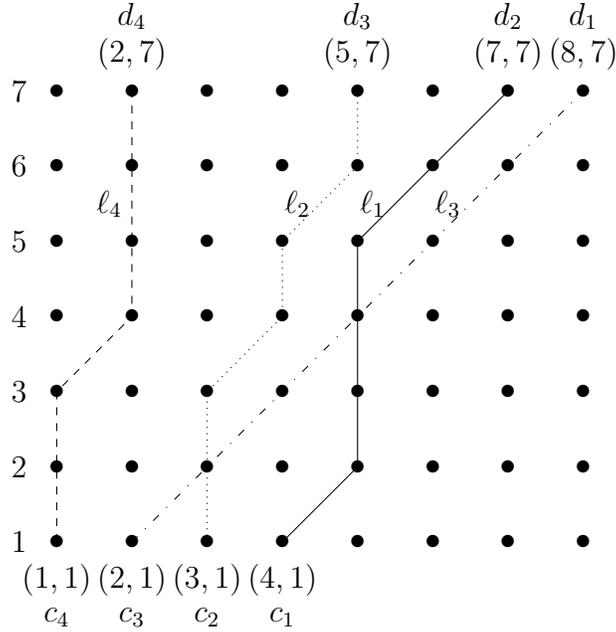
\end{Example}
 
 Put $S_E^{\lambda}=\{{\rm{type}}(L)\in {\frak S}_s | L\in {\mathcal E}_{\lambda}^N\}$,
then the bijection map $\tau_E : {\rm{Rim}}_E^{\lambda}\to S_E^{\lambda}$ given by $\tau_E(\Theta)={\rm{type}}(L)$ exists. (cf. \cite[Lemma 3.12]{NPY}).

 \begin{Lemma}\label{Lem31}
 For any partition $\lambda=(\lambda_1, \ldots, \lambda_r)$ and its conjugate $\lambda'=(\lambda_1', \ldots, \lambda'_s)$,
 let $\Theta^{\sigma}=(\theta_1^{\sigma}, \ldots, \theta_s^{\sigma})\in {\rm{Rim}}_E^{\lambda}$ be the $E$-rim decomposition such that $\tau_E(\Theta^{\sigma})=\sigma$ for $\sigma\in S_{E}^{\lambda}$. 
For $\Theta=(\theta_1, \ldots, \theta_s)\in {\rm{Rim}}_E^{\lambda}$, $\theta_i({\pmb s})\in {\mathbb C}^{|\theta_i|}$ is the tuple obtained by reading contents of the shape restriction of ${\pmb s}$ to $\theta_i$ from the top right to the bottom left.
And let $\varepsilon_{\sigma}$ be the signature of $\sigma\in {\frak S}_s$.
Set $$W_{\lambda, E}=W_{\lambda, E}({\mathbb C})=\left\{{\pmb s}=(s_{ij})\in T_{\lambda}( {\mathbb C}) \left| 
\begin{matrix}
\Re(s_{ij})\geq 1 \; {\rm{for\; all}}\; (i, j)\in D_{\lambda}\backslash E_{\lambda}\\
\hspace{-7mm}\Re(s_{ij})> 1 \; {\rm{for\; all}}\; (i, j)\in E_{\lambda}\\
\end{matrix}
\right. \right\},$$
where 
$E_{\lambda}=\{(i, j)\in D_{\lambda} | i-j\in \{i-\lambda_i' | 1\leq i \leq s\}\}$.
When ${\pmb s}\in W_{\lambda, E}$,
 we have
 \begin{equation}\label{keylemma2}
 \sum_{diag}\zeta_{\lambda}({\pmb s})=\sum_{diag}\sum_{\sigma\in S_E^{\lambda}}\varepsilon_{\sigma}\prod_{i=1}^s \zeta(\theta_i^{\sigma}({\pmb s})),
 \end{equation}
 where $\sum_{diag}=\sum_{\substack{\sigma_j\in {S}_j\\ j\in {\mathbb Z}}}\prod_{i\in {\mathbb Z}}\sigma_i$ for $S_j$ being the set of permunation of the elements of $I(J)=\{(k, \ell)\in D_{\lambda} | \ell-k=j\}$.
 \end{Lemma}
 
 \begin{proof}
The assertion is from similar calculation as in Lemma 3.1 in \cite{NT}.
 \end{proof}

If all the elements in each diagonal line in ${\pmb s}$ are equal,  $\sum_{diag}$ in \eqref{keylemma2} doesn't make sense and
we have following 
\begin{Corollary}\label{Cor31}
Set $W_{\lambda}^{\mathrm{diag}}=\{{\pmb s}=(s_{ij})\in W_{\lambda} | s_{ij}=s_{pq}\; {\rm{if}}\; j-i=q-p\}$.
 When ${\pmb s}\in W_{\lambda}^{\mathrm{diag}}$,
 we have
 \begin{equation}\label{keylemma1}
\zeta_{\lambda}({\pmb s})=\sum_{\sigma\in S_E^{\lambda}}\varepsilon_{\sigma}\prod_{i=1}^s \zeta(\theta_i^{\sigma}({\pmb s})).
\end{equation}
\end{Corollary}

 The following Jacobi-Trudi formula for Schur multiple zeta function under suitable assumption on variables is obtained by Corollary \ref{Cor31}.
 
 \begin{Theorem}(\cite[Theorem1.1(2)]{NPY})
Let $\lambda=(\lambda_1, \ldots, \lambda_r)$ be a partition and $\lambda'=(\lambda_1', \ldots, \lambda'_s)$ be the conjugate of $\lambda$. 
Put
 $$W_{\lambda}^{\rm JT}=\left\{{\pmb s}=(s_{ij})\in W_{\lambda}^{\mathrm{diag}} 
 \left| 
\; \Re(s_{\lambda_i', i)}>1 \; {\rm for}\; 1\leq{}^{\forall}i\leq s
 \right.\right\}$$
and we write $a_k=s_{i, i+k}$ for $k\in {\mathbb Z}$ (and for any $i\in {\mathbb N}$).
Then we have
\begin{equation}\label{JTE}
\zeta_{\lambda}({\pmb s})= \det \left[ \zeta (a_{j-1}, a_{j-2}, \ldots, a_{j-(\lambda'_i-i+j)})\right]_{1\leq i, j\leq s}.
\end{equation}
Here, we understand that $\zeta(\cdots)=1$ if $\lambda'_i-i+j=0$ and $0$ if $\lambda'_i-i+j<0$.
\end{Theorem} 

For skew type Schur multiple zeta functions, we have the similar formula:
\begin{Lemma}\label{Lemdagger}
Let $\delta=\lambda/\mu$.
For ${\pmb s}
\in W_{\delta, E}({\mathbb C})$, we have
$$
\sum_{diag}\zeta_{\delta} ({\pmb s})
=\sum_{diag}\sum_{\sigma\in S_E^{\delta}}\varepsilon_{\sigma}\prod_{i=1}^s \zeta(\theta_i^{\sigma}({\pmb s})).
$$
\end{Lemma}

 \begin{Theorem}\label{skewJT}(\cite[Theorem4.3(2)]{NPY})
Retain the above notations. Assume that ${\pmb s}=(s_{ij})\in W_{\delta}^{\rm{JT}}$ and we write $a_k=s_{i, i+k}$ for $k\in {\mathbb Z}$ (and for any $i\in {\mathbb N}$).
Then we have
\begin{equation}\label{JTEskew}
\zeta_{\delta}({\pmb s})= \det \left[ \zeta (a_{-\mu_j'+j-1}, a_{-\mu_j'+j-2}, \ldots, a_{-\mu_j'+j-(\lambda'_i-i+j)})\right]_{1\leq i, j\leq s}.
\end{equation}
Here, we understand that $\zeta(\cdots)=1$ if $\lambda'_i-\mu_j'-i+j=0$ and $\zeta(\cdots)=0$ if $\lambda'_i-\mu_j'-i+j<0$ or
at least one $a_n$ doesn't exist ($-\mu_j'+j-1\geq n \geq \mu_j'+j-(\lambda'-i+j)$) among the variants in the $ij$-th element.
\end{Theorem}

\section{Duality formula}\label{sectionresults}
For positive integers $k_1, \ldots, k_n$ with $k_n\geq 2$, the multiple zeta values $\zeta(k_1, \ldots, k_n)$ converges absolutely and such an index set ${\bf k}=(k_1, \ldots, k_n)$ is called {\it admissible} index set. When we write an admissible index set 
${\bf k}$ as
$${\bf k}=(\underbrace{1, \ldots, 1}_{a_1-1}, b_1+1, \underbrace{1, \ldots, 1}_{a_2-1}, b_2+1, \ldots, \underbrace{1, \ldots, 1}_{a_m-1}, b_m+1)
$$
with $a_1, b_1, a_2, b_2, \cdots, a_m, b_m\in {\mathbb Z}_{\geq 1}$, the following index set is called {\it dual} index set of ${\bf k}$:
$${\bf k^{\dagger}}=(\underbrace{1, \ldots, 1}_{b_m-1}, a_m+1, \underbrace{1, \ldots, 1}_{b_{m-1}-1}, a_{m_1}+1, \ldots, \underbrace{1, \ldots, 1}_{b_1-1}, a_1+1).
$$

Then the well-known duality relation for multiple zeta values is as follows.
\begin{Theorem}\label{Thduality}
For any admissible index set ${\bf k}$ and its dual index set ${\bf k}^{\dagger}$, we have
$$\zeta({\bf k})=\zeta({\bf k}^{\dagger}).
$$
\end{Theorem}

Inspired by this formula, we will define a tableau which is ``dual" to ${\pmb k}\in T_{\delta}({\mathbb Z})$ with
$\lambda=(\lambda_1, \ldots, \lambda_r)$ and $\mu=(\mu_1, \ldots, \mu_r)$ being two partitions such that $\lambda_i\geq \mu_i$ for any $i$.
First, we denote a finer piece of index $\underbrace{1, \ldots, 1}_{a-1}, b+1$ as $A(a,b)$ and call it {\it admissible piece}. Then if we write $A_i:=A(a_i, b_i)$
and $A_i^{\dagger}:=A(b_i, a_i)$,  
above
admissible index set ${\bf k}$ and its dual ${\bf k}^{\dagger}$ are written in terms of admissible pieces:
$${\bf k}=(A(a_1, b_1), A(a_2, b_2), \ldots, A(a_m, b_m))=(A_1, A_2, \ldots, A_m)
$$
and
$${\bf k^{\dagger}}=(A(b_m, a_m), A(b_{m-1}, a_{m-1}), \ldots, A(b_1, a_1))=(A_m^{\dagger}, A_{m-1}^{\dagger}, \ldots, A_1^{\dagger}).
$$
 
Let $T^{\mathrm{diag}}_\delta({\mathbb Z})=\{{\pmb s}\in T_{\delta}({\mathbb Z}) | s_{ij}=s_{pq} \; {\mathrm{if}}\; j-i=q-p\}$. 
 Put $I_{\delta}^D$ to be the set of elements in $T_{\delta}^{\mathrm{diag}}({\mathbb Z})$ consisted of
admissible pieces such that the right side of the top element in each column is not $1$. 
For ${\pmb s}\in I_{\delta}^D$, in terms of admissible pieces, 
the row which has topmost component is numbered as the first row. 
We write the component in the $i$-th row and $j$-th column in terms of admissible pieces as $A_{ij}$. 
We note that
the component in the upper-right corner which is in the $j$-th row is $A_{1j}$
and that $A_{ij}=A_{k\ell}$ if $j-i=\ell-k$ when they are not empty. 
Further, we notice that in terms of tableau,
the top element in $A_{ij}$ and the bottom element in $A_{i(j+1)}$ are located side by side.

Next, we write ${\pmb k}\in I_{\delta}^{\rm D}$ as
\begin{equation}\label{notationk}
{\pmb k}={\pmb k}^{\mathrm{col}}_1\cdots {\pmb k}^{\mathrm{col}}_{\lambda_1},
\end{equation}
where ${\pmb k}^{\mathrm{col}}_{j}$ is the $j$-th column tableau of ${\pmb k}$. For example, when $\lambda=(3, 2, 1)$ and
$${\pmb k}=
\ytableausetup{boxsize=normal}  
\begin{ytableau}
 k_{11}   & k_{12} & k_{13}\\
 k_{21}  & k_{22}\\
 k_{31}
\end{ytableau},
$$
then 
${\pmb k}^{\mathrm{col}}_{1}=
\begin{ytableau}
 k_{11} \\
 k_{21} \\
 k_{31}
\end{ytableau}
$,
${\pmb k}^{\mathrm{col}}_{2}=
\begin{ytableau}
 k_{12} \\
 k_{22} 
\end{ytableau}
$
and
${\pmb k}^{\mathrm{col}}_{3}=
\begin{ytableau}
 k_{13} 
\end{ytableau}
$.
We regard ${\pmb k}^{\mathrm{col}}_{j}$ ($1\leq j\leq \lambda_1$) as the corresponding admissible index set in order from top to bottom.
So, we may write it in terms of admissible pieces as well. 
If the $j$-th column tableau ${\pmb k}^{\mathrm{col}}_{j}$ starts $A_{nj}$ for some $n$ and  has $m+1$ admissible pieces,
then ${\pmb k}^{\mathrm{col}}_{j}={}^t (A_{nj}, \ldots A_{(n+m) j})$. 
Then the dual tableau is ${\pmb k}^{\mathrm{col}, {\dagger}}_{j}={}^t (A_{(n+m)j}^{\dagger}, \ldots A_{nj}^{\dagger})$.
We define ${\pmb k}^{\dagger}$ by arranging  ${\pmb
 k}^{\mathrm{col}, {\dagger}}_{\lambda_1}, \ldots, {\pmb k}^{\mathrm{col}, {\dagger}}_{1}$ in this order from left to right,
where we put the top element in $A_{ij}^{\dagger}$ and the bottom element in $A_{i(j-1)}^{\dagger}$ side by side
for $2\leq j\leq \lambda_1$ if both $A_{ij}^{\dagger}$ and $A_{i(j-1)}^{\dagger}$ are not empty.

\begin{Example}
For $\delta=\lambda/\mu$ with $\lambda=(3,2,1)$ and $\mu=\emptyset$,
let ${\pmb k}=
\begin{ytableau}
\ytableausetup{centertableaux}
 2 & 2 & 3 \\
4 & 2\\
5 
\end{ytableau}$. Then 
${\pmb k}=
{\pmb k}^{\mathrm{col}}_{1}{\pmb k}^{\mathrm{col}}_{2}{\pmb k}^{\mathrm{col}}_{3},
$
where 
${\pmb k}^{\mathrm{col}}_{1}=
\begin{ytableau}
2\\
4\\
5 
\end{ytableau}
$,
${\pmb k}^{\mathrm{col}}_{2}=
\begin{ytableau}
2\\
2 
\end{ytableau}
$
and
${\pmb k}^{\mathrm{col}}_{3}=
\begin{ytableau}
3
\end{ytableau}
$.
In terms of admissible pieces, 
${\pmb k}^{\mathrm{col}}_{1}={}^t(A(1,1), A(1,3), A(1,4))={}^t(A_{11}, A_{21}, A_{31})$, 
${\pmb k}^{\mathrm{col}}_{2}={}^t(A(1,1), A(1,1))={}^t(A_{12}, A_{22})$ and
${\pmb k}^{\mathrm{col}}_{3}={}^t(A(1,2))={}^t(A_{13})$ and
${\pmb k}=
\begin{ytableau}
 A_{11} & A_{12} & A_{13} \\
A_{21} & A_{23}\\
A_{31} 
\end{ytableau}$
,
where 
$$
\begin{array}{|c|c|c|c|}
\hline
 & A(a_i,b_i) & A(a_i,b_i)^{\dagger} & {\rm pieces\; in}\; {\pmb k}\\
\hline
A(1,1) & 2 & 2 & A_{11}, A_{12}, A_{22}\\
\hline
A(1,2) & 3 & 1, 2 & A_{13}\\
\hline
A(1,3) & 4 & 1, 1, 2 & A_{21}\\
\hline
A(1,4) & 5 & 1, 1, 1, 2 & A_{31}\\
\hline
\end{array}.
$$
The dual tableau of 
${\pmb k}^{\mathrm{col}, {\dagger}}_{j}$ are
${\pmb k}^{\mathrm{col}, {\dagger}}_{3}=
\begin{ytableau}
1\\
2 
\end{ytableau}
=
\begin{ytableau}
A_{13}^{\dagger}
\end{ytableau}
$,
${\pmb k}^{\mathrm{col}, {\dagger}}_{2}=
\begin{ytableau}
2\\
2 
\end{ytableau}
=
\begin{ytableau}
A_{22}^{\dagger}\\
A_{12}^{\dagger}
\end{ytableau}
$
and
${\pmb k}^{\mathrm{col}, {\dagger}}_{3}=
\begin{ytableau}
1\\
1\\1\\2\\
1\\1\\2\\2
\end{ytableau}
=
\begin{ytableau}
A_{31}^{\dagger}\\
A_{21}^{\dagger}\\
A_{11}^{\dagger}\\
\end{ytableau}
$. And so, the dual tableau 
${\pmb k}^{\dagger}=
\begin{ytableau}
\none & \none & A_{31}^{\dagger}\\
\none & A_{22}^{\dagger} & A_{21}^{\dagger}\\
A_{13}^{\dagger} & A_{12}^{\dagger} & A_{11}^{\dagger}\\
\end{ytableau}
=
\begin{ytableau}
\none & \none & 1\\
\none & \none & 1\\
\none & \none & 1\\
\none & \none & 2\\
\none & \none & 1\\
\none & \none & 1\\
\none & 2 & 2\\
1& 2 & 2\\
2
\end{ytableau}.
$
\end{Example}

\begin{Example}
For $\delta=\lambda/\mu$ with $\lambda=(3,2,1)$ and $\mu=(1, 1)$,
let ${\pmb k}=
\begin{ytableau}
 \none & 1 & 3 \\
4 & 2\\
5 
\end{ytableau}$. Then 
${\pmb k}=
\begin{ytableau}
A_{11} & A_{12} & A_{13} \\
A_{21}
\end{ytableau}$
and  the dual tableau is
${\pmb k}^{\dagger}=
\begin{ytableau}
\none & \none & A_{21}^{\dagger}\\
A_{13}^{\dagger} & A_{12}^{\dagger} & A_{11}^{\dagger}\\
\end{ytableau}
=
\begin{ytableau}
\none & \none & 1\\
\none & \none & 1\\
\none & \none & 1\\
\none & \none & 2\\
\none & \none & 1\\
\none & \none & 1\\
1 & 3 & 2\\
2 
\end{ytableau}.
$
\end{Example}

\begin{Theorem}\label{maintheorem}
Let $\lambda$ and $\mu$ be partitions. Put $\delta=\lambda/\mu$ and ${\pmb k} \in I_{\delta}^{\rm D}$.
For ${\pmb k}^{\dagger}$ being the dual tableau of ${\pmb k}$ and $\delta^{{\dagger}}$ being the shape of ${\pmb k}^{\dagger}$, we have
$$
\zeta_{\delta}({\pmb k})=\zeta_{\delta^{\dagger}}({\pmb k}^{\dagger}).
$$
\end{Theorem}
We can say this is an extension of  \cite[Corollary 6.2]{NPY}.\\

\noindent
\begin{proof}
By the duality formula Theorem \ref{Thduality}, \eqref{JTEskew} equals
$$
\zeta_{\delta}({\pmb s})= \det \left[ \zeta (a_{-\mu_j'+j-1}, a_{-\mu_j'+j-2}, \ldots, a_{-\mu_j'+j-(\lambda'_i-\mu'_j-i+j)})^{\dagger}\right]_{1\leq i, j\leq s}.
$$
Replacing $i$ and $j$ with $s-j+1$ and $s-i+1$, respectively.
Then the property of the determinant of matrix leads to
$$ \det \left[ \zeta (a_{-\mu_{s-i+1}'+s-i}, a_{-\mu_{s-i+1}'+s-i-1}, \ldots, a_{-\mu_{s-i+1}'+s-i+1-(\lambda'_{s-j+1}-\mu'_{s-i+1}+j-i)})^{\dagger}\right]_{1\leq i, j\leq s}.
$$

Using the Jacobi-Trudi formula for Schur multiple zeta function of skew type \eqref{JTEskew} (\cite[Theorem 4.3 (2)]{NPY}), we have the desired relation.
\end{proof}
\begin{Example}
For $\delta=\lambda/\mu$ with $\lambda=(3,2,1)$ and $\mu=(1,1)$, 
let ${\pmb k}=
\begin{ytableau}
 \none & 1 & 3 \\
\none & 2\\
5 
\end{ytableau}$. Then 
${\pmb k}=
\begin{ytableau}
\none & A_{12} & A_{13} \\
A_{21}
\end{ytableau}$
and  the dual tableau is
${\pmb k}^{\dagger}=
\begin{ytableau}
\none & \none & A_{21}^{\dagger}\\
A_{13}^{\dagger} & A_{12}^{\dagger} 
\end{ytableau}
=
\begin{ytableau}
\none & \none & 1\\
\none & \none & 1\\
\none & \none & 1\\
\none & \none & 2\\
1 & 3 \\
2 
\end{ytableau}.
$
We can see
$$\zeta_{\delta}(\pmb k)=
\left|
\begin{matrix}
\begin{ytableau}
5
\end{ytableau} & 0 & 0\\
0 & \begin{ytableau}
1\\
2
\end{ytableau} & 
\begin{ytableau}
3\\
1\\
2
\end{ytableau}\vspace{2mm}
\\
0 & 1 & \begin{ytableau}
3
\end{ytableau}
\end{matrix}
\right|
=
\left|
\begin{matrix}
\begin{ytableau}
1\\
1\\
1\\
2
\end{ytableau} & 0 & 0\\
0 & \begin{ytableau}
3\end{ytableau} & 
\begin{ytableau}
3\\
1\\
2
\end{ytableau}\vspace{2mm}
\\
0 & 1 & \begin{ytableau}
1\\
2
\end{ytableau}
\end{matrix}
\right|
=
\left|
\begin{matrix}
\begin{ytableau}
1\\
2
\end{ytableau}
& 
\begin{ytableau}
3\\
1\\
2
\end{ytableau}\vspace{2mm}
&
0\\
1 &
\begin{ytableau}
3\end{ytableau} & 
0\\
0 & 0 &
\begin{ytableau}
1\\
1\\
1\\
2
\end{ytableau} 
\end{matrix}
\right|
=\zeta_{{\delta}^{\dagger}}({\pmb k}^{\dagger}),
$$
where ${\delta}^{\dagger}$ is the shape of ${\pmb k}^{\dagger}$. 
\end{Example}

\section{Ohno relation}\label{sectionresults2}
 Ohno  \cite{O} showed a following family of ${\mathbb Q}$-linear relations among multiple zeta values called
Ohno relation:
\begin{Theorem}\label{Ohnorelation}
For any $\ell\in {\mathbb Z_{\geq 0}}$ and
any admissible index set ${\bf k}=(k_1, \ldots, k_r)$ and its dual index set ${\bf k^{\dagger}}=(k^{\dagger}_1., \ldots, k^{\dagger}_s)$,
\begin{equation}\label{ohno}
\sum_{\substack{|\varepsilon|=\varepsilon_1+\cdots+\varepsilon_r=\ell\\
{}^{\forall}\varepsilon_i\geq 0}}\zeta(k_1+\varepsilon_1, \ldots, k_r+\varepsilon_r)=
\sum_{\substack{|\varepsilon'|=\varepsilon'_1+\cdots+\varepsilon'_s=\ell\\
{}^{\forall}\varepsilon\geq 0}} \zeta(k^{\dagger}_1+\varepsilon'_1, \ldots, k^{\dagger}_s+\varepsilon'_s).
\end{equation}
\end{Theorem}

Let $\lambda=(\lambda_1, \ldots, \lambda_r)$ and $\mu=(\mu_1, \ldots, \mu_r)$ be two partitions such that 
$\lambda_i\geq \mu_i$ for all $i$. 
Put $\delta=\lambda/\mu$. For ${\pmb k}\in W_\delta({\mathbb Z}_{\geq 1})$ and
${\pmb \varepsilon}\in T_{\delta}({\mathbb Z}_{\geq 0})$, we denote by 
$${\mathcal O}({\pmb k}: \ell):=\sum_{\substack{|{\pmb \varepsilon}|=\ell}} \zeta_{\delta} ({\pmb k}+{\pmb \varepsilon})
$$
for $\ell\in {\mathbb Z_{\geq 0}}$.
For ${\pmb k}_{j}^{\mathrm{col}}\in I_{(1^n)}^D$ ($j_1\leq j\leq j_r$)'s,  we define
$$
{\mathcal O}({\pmb k}_{j_1}^{\mathrm{col}}\times {\pmb k}_{j_2}^{\mathrm{col}}\times\cdots \times {\pmb k}_{j_r}^{\mathrm{col}} : \ell) 
:=\sum_{\ell_1+\ell_2+\cdots+\ell_r=\ell}
{\mathcal O}({\pmb k}_{j_1}^{\mathrm{col}} : \ell_1) \cdots {\mathcal O}({\pmb k}_{j_r}^{\mathrm{col}} : \ell_r).
$$ 
Then we have following
\begin{Lemma}\label{Lemprodzeta}
For ${\pmb k}_{j}^{\mathrm{col}}\in I_{(1^n)}^D$ ($j_1\leq j\leq j_r$),  we have
$$
{\mathcal O}({\pmb k}_{j_1}^{\mathrm{col}}\times {\pmb k}_{j_2}^{\mathrm{col}}\times\cdots \times {\pmb k}_{j_r}^{\mathrm{col}} : \ell) =
{\mathcal O}({\pmb k}_{j_r}^{\mathrm{col}, \dagger}\times \cdots \times {\pmb k}_{j_2}^{\mathrm{col}, \dagger}\times {\pmb k}_{j_1}^{\mathrm{col}, \dagger}: \ell) .
$$
\end{Lemma}
\begin{proof}
By Theorem \ref{Ohnorelation}, we have
\begin{align*}
{\mathcal O}({\pmb k}_{j_1}^{\mathrm{col}}\times {\pmb k}_{j_2}^{\mathrm{col}}\times\cdots \times {\pmb k}_{j_r}^{\mathrm{col}} : \ell) 
&=\sum_{\ell_1+\ell_2+\cdots+\ell_r=\ell}
{\mathcal O}({\pmb k}_{j_1}^{\mathrm{col}} : \ell_1) \cdots {\mathcal O}({\pmb k}_{j_r}^{\mathrm{col}} : \ell_r)\\
&=\sum_{\ell_1+\ell_2+\cdots+\ell_r=\ell}
{\mathcal O}({\pmb k}_{j_1}^{\mathrm{col}, \dagger} : \ell_1) \cdots {\mathcal O}({\pmb k}_{j_r}^{\mathrm{col}, \dagger} : \ell_r)\\
&=\sum_{\ell_1+\ell_2+\cdots+\ell_r=\ell}
{\mathcal O}({\pmb k}_{j_r}^{\mathrm{col}, \dagger} : \ell_r) \cdots {\mathcal O}({\pmb k}_{j_1}^{\mathrm{col}, \dagger} : \ell_1)\\
&={\mathcal O}({\pmb k}_{j_r}^{\mathrm{col}, \dagger}\times {\pmb k}_{j_{r-1}}^{\mathrm{col}, \dagger}\times\cdots \times {\pmb k}_{j_1}^{\mathrm{col}, \dagger} : \ell) .
\end{align*}

\end{proof}
Applying Lemma \ref{Lemprodzeta}, we obtain the following formula which is the extension of Ohno relation(Theorem \ref{Ohnorelation}). 
\begin{Theorem}\label{Thmohno}
Let $\lambda$ and $\mu$ be partitions. Put $\delta=\lambda/\mu$ and ${\pmb k} \in I_{\delta}^{\rm D}$
and ${\pmb k}^{\dagger}$ being the dual tableau of ${\pmb k}$, for $\ell\in {\mathbb Z_{\geq 0}}$ we have
$$
{\mathcal O}({\pmb k}: \ell)={\mathcal O}({\pmb k}^{\dagger}: \ell).
$$
\end{Theorem}
We will explain it with the example, first.
\begin{Example}
Let $\lambda=(2,2)$ and $\mu=\emptyset$. We consider ${\pmb k}=
\begin{ytableau}
2 & 3\\
4 &2
\end{ytableau} 
\in I_{\lambda}^D$ and ${\pmb \varepsilon}\in T_{\lambda}({\mathbb Z}_{\geq 0})$
with $\ell=1$. Then
$$\left\{{\pmb k}+{\pmb \varepsilon}\; \left| \;|{\pmb \varepsilon}|=\ell| \right. \right\}=
\left\{
\begin{ytableau}
3 & 3\\
4 &2
\end{ytableau} , \;
\begin{ytableau}
2 & 4\\
4 &2
\end{ytableau} ,\;
\begin{ytableau}
2 & 3\\
5 &2
\end{ytableau} ,\;
\begin{ytableau}
2 & 3\\
4 &3
\end{ytableau} 
\right\}.
$$ 
By Lemma \ref{Lem31}, we have
\begin{align}
&\sum_{\mathrm{diag}}
\zeta_{\lambda}\left(
\begin{ytableau}
3 & 3\\
4 &2
\end{ytableau}
\right)=
\zeta_{\lambda}\left(
\begin{ytableau}
3 & 3\\
4 &2
\end{ytableau}
\right)+
\zeta_{\lambda}\left(
\begin{ytableau}
2 & 3\\
4 &3
\end{ytableau}
\right)=\sum_{\mathrm{diag}}\sum_{\sigma\in S_E^{\lambda}}\varepsilon_{\sigma}\prod_{i=1}^2 \zeta\left(\theta_i^{\sigma}\left(
\begin{ytableau}
3 & 3\\
4 &2
\end{ytableau}
\right)\right)\label{diag1},\\
&
\sum_{\mathrm{diag}}
\zeta_{\lambda}\left(
\begin{ytableau}
2 & 4\\
4 &2
\end{ytableau}
\right)=
\zeta_{\lambda}\left(
\begin{ytableau}
2 & 4\\
4 &2
\end{ytableau}
\right)+
\zeta_{\lambda}\left(
\begin{ytableau}
2 & 4\\
4 &2
\end{ytableau}
\right)=2\sum_{\sigma\in S_E^{\lambda}}\varepsilon_{\sigma}\prod_{i=1}^2 \zeta\left(\theta_i^{\sigma}\left(
\begin{ytableau}
2 & 4\\
4 &2
\end{ytableau}
\right)\right)\label{diag2},\\
&
\sum_{\mathrm{diag}}
\zeta_{\lambda}\left(
\begin{ytableau}
2 & 3\\
5 &2
\end{ytableau}
\right)=
\zeta_{\lambda}\left(
\begin{ytableau}
2 & 3\\
5 &2
\end{ytableau}
\right)+
\zeta_{\lambda}\left(
\begin{ytableau}
2 & 3\\
5 &2
\end{ytableau}
\right)=2\sum_{\sigma\in S_E^{\lambda}}\varepsilon_{\sigma}\prod_{i=1}^2 \zeta\left(\theta_i^{\sigma}\left(
\begin{ytableau}
2 & 3\\
5 &2
\end{ytableau}
\right)\right)\label{diag3},\\
&\sum_{\mathrm{diag}}
\zeta_{\lambda}\left(
\begin{ytableau}
2 & 4\\
4 &3
\end{ytableau}
\right)=
\zeta_{\lambda}\left(
\begin{ytableau}
2 & 3\\
4 &3
\end{ytableau}
\right)+
\zeta_{\lambda}\left(
\begin{ytableau}
3 & 3\\
4 &2
\end{ytableau}
\right)=\sum_{\mathrm{diag}}\sum_{\sigma\in S_E^{\lambda}}\varepsilon_{\sigma}\prod_{i=1}^2 \zeta\left(\theta_i^{\sigma}\left(
\begin{ytableau}
2 & 3\\
4 &3
\end{ytableau}
\right)\right)\label{diag4}.
\end{align}
Since the equations \eqref{diag1} and \eqref{diag4} are the same, taking the summation both sides of them leads to
\begin{align}
2{\mathcal O}\left({\pmb k}:1\right)
&=
\sum_{\mathrm{diag}}\left(
\zeta_{\lambda}\left(
\begin{ytableau}
3 & 3\\
4 &2
\end{ytableau}
\right)
+
\zeta_{\lambda}\left(
\begin{ytableau}
2 & 4\\
4 &2
\end{ytableau}
\right)
+
\zeta_{\lambda}\left(
\begin{ytableau}
2 & 3\\
5 &2
\end{ytableau}
\right)
+
\zeta_{\lambda}\left(
\begin{ytableau}
2 & 3\\
4 &3
\end{ytableau}
\right)
\right)\notag\\
&=2\sum_{\sigma\in S_E^{\lambda}}\varepsilon_{\sigma}\prod_{i=1}^2 
\left(
\zeta\left(\theta_i^{\sigma}\left(
\begin{ytableau}
3 & 3\\
4 &2
\end{ytableau}
\right)\right)
+
\zeta\left(\theta_i^{\sigma}\left(
\begin{ytableau}
2 & 4\\
4 &2
\end{ytableau}
\right)\right)
\right.\notag\\
&\hspace{1cm}+\left.
\zeta\left(\theta_i^{\sigma}\left(
\begin{ytableau}
2 & 3\\
5 &2
\end{ytableau}
\right)\right)
+
\zeta\left(\theta_i^{\sigma}\left(
\begin{ytableau}
2 & 3\\
4 &3
\end{ytableau}
\right)\right)
\right)\label{examplecal},
\end{align}
where the coefficient $2$ is the order of symmetric group whose order is the number of the elements in the diagonal. Dividing by $2$ in \eqref{examplecal}, we obtain
$${\mathcal O}\left({\pmb k}:1\right)=\sum_{|{\pmb \varepsilon}|=1}\sum_{\sigma\in S_E^{\lambda}}\varepsilon_{\sigma}\prod_{i=1}^2\zeta(\theta_i^{\sigma}({\pmb k}+{\pmb \varepsilon})).
$$
The right hand side is
\begin{align*}
&\sum_{\sigma\in S_E^{\lambda}}\varepsilon_{\sigma}\prod_{i=1}^2 
\left(
\zeta\left(\theta_i^{\sigma}\left(
\ytableausetup{centertableaux}
\begin{ytableau}
3 & 3\\
4 &2
\end{ytableau}
\right)\right)
+
\zeta\left(\theta_i^{\sigma}\left(
\begin{ytableau}
2 & 4\\
4 &2
\end{ytableau}
\right)\right)
\right.\\
&\hspace{4cm}+\left.
\zeta\left(\theta_i^{\sigma}\left(
\begin{ytableau}
2 & 3\\
5 &2
\end{ytableau}
\right)\right)
+
\zeta\left(\theta_i^{\sigma}\left(
\begin{ytableau}
2 & 3\\
4 &3
\end{ytableau}
\right)\right)
\right)\\
\\
&=\zeta(3,4)\zeta(3,2)+
\zeta(2,4)\zeta(4,2)+\zeta(2,5)\zeta(3,2)+\zeta(2,4)\zeta(3,3)\\
&-(\zeta(3)\zeta(3,2,4)+\zeta(2)\zeta(4,2,4)+\zeta(2)\zeta(3,2,5)
+\zeta(2)\zeta(3,3,4))
\\
&=
{\mathcal O}\left(
\begin{ytableau}
2 \\
4 
\end{ytableau}\times
\begin{ytableau}
3 \\
2 
\end{ytableau}\; : 1\right)
-
{\mathcal O}\left(
\begin{ytableau}
2
\end{ytableau}\times
\begin{ytableau}
3 \\
2 \\
4 
\end{ytableau}\; : 1\right).
\end{align*}
By Lemma \ref{Lemprodzeta}, this becomes
$$
{\mathcal O}\left(
\ytableausetup{centertableaux}
\begin{ytableau}
2 \\
1\\
2
\end{ytableau}\times
\begin{ytableau}
1 \\
1\\
2\\
2
\end{ytableau}\; : 1\right)
-
{\mathcal O}\left(
\begin{ytableau}
1 \\
1\\
2 \\
2\\
1\\
2 
\end{ytableau}\times
\begin{ytableau}
2
\end{ytableau}\; : 1\right).
$$
Similar calculation above gives
\begin{align*}
&\sum_{\sigma\in S_E^{\lambda}}\varepsilon_{\sigma}\prod_{i=1}^2 
\left(
\zeta\left(\theta_i^{\sigma}\left(
\begin{ytableau}
\none & 1\\
\none & 1\\
3 & 2\\
1 & 2\\
2
\end{ytableau}
\right)\right)
+
\zeta\left(\theta_i^{\sigma}\left(
\begin{ytableau}
\none & 1\\
\none & 1\\
2 & 2\\
2 & 2\\
2
\end{ytableau}
\right)\right)
+
\zeta\left(\theta_i^{\sigma}\left(
\begin{ytableau}
\none & 1\\
\none & 1\\
2 & 2\\
1 & 2\\
3
\end{ytableau}
\right)\right)
\right.\\
&+\left.
\zeta\left(\theta_i^{\sigma}\left(
\begin{ytableau}
\none & 2\\
\none & 1\\
2 & 2\\
1 & 2\\
2
\end{ytableau}
\right)\right)
+
\zeta\left(\theta_i^{\sigma}\left(
\begin{ytableau}
\none & 1\\
\none & 2\\
2 & 2\\
1 & 2\\
2
\end{ytableau}
\right)\right)
+
\zeta\left(\theta_i^{\sigma}\left(
\begin{ytableau}
\none & 1\\
\none & 1\\
2 & 3\\
1 & 2\\
2
\end{ytableau}
\right)\right)
+
\zeta\left(\theta_i^{\sigma}\left(
\begin{ytableau}
\none & 1\\
\none & 1\\
2 & 2\\
1 & 3\\
2
\end{ytableau}
\right)\right)
\right)\\
&={\mathcal O}\left(
\ytableausetup{centertableaux}
\begin{ytableau}
\none & 1\\
\none & 1\\
2 & 2\\
1 & 2\\
2
\end{ytableau}
\; :1
\right)={\mathcal O}({\pmb k}^{\dagger}\; : 1).
\end{align*}
 \end{Example}
\noindent
{\it Proof of Theorem \ref{Thmohno}.}\;
Assume ${\pmb k}$ has $n$ diagonal lines with $2$ or more elements.
Each of them has $m_1, \ldots, m_n$ elements, respectively ($m_i\geq 2$ for $1\leq i\leq n$).

By Lemma \ref{Lem31},
\begin{align*}
|{\frak S}_{m_1}|\ldots |{\frak S}_{m_n}|{\mathcal O}({\pmb k}: \ell)
&=\sum_{\mathrm{diag}}\sum_{\substack{|{\pmb \varepsilon}|=\ell}} \zeta_{\delta} ({\pmb k}+{\pmb \varepsilon})\\
&=|{\frak S}_{m_1}|\ldots |{\frak S}_{m_n}|
\sum_{\substack{|{\pmb \varepsilon}|=\ell}}\sum_{\sigma\in S_E^{\delta}}\varepsilon_{\sigma}\prod_{i=1}^s \zeta(\theta_i^{\sigma}({\pmb k}+{\pmb \varepsilon})).
\end{align*}
This leads to
$${\mathcal O}({\pmb k}: \ell)=\sum_{\substack{|{\pmb \varepsilon}|=\ell}}\sum_{\sigma\in S_E^{\delta}}\varepsilon_{\sigma}\prod_{i=1}^s \zeta(\theta_i^{\sigma}({\pmb k}+{\pmb \varepsilon})).
$$
After changing the order of summations, 
\begin{align*}
{\mathcal O}({\pmb k}: \ell)&=\sum_{\sigma\in S_E^{\delta}}\varepsilon_{\sigma}\sum_{\substack{
{\pmb \varepsilon}\in T_{\delta}({\mathbb Z}_{\geq 0})\\
|{\pmb \varepsilon}|=\ell}}\prod_{i=1}^s \zeta(\theta_i^{\sigma}({\pmb k}+{\pmb \varepsilon}))\\
&=\sum_{\sigma\in S_E^{\lambda}}\varepsilon_{\sigma}{\mathcal O}(\theta_1^{\sigma}({\pmb k})\times \theta_2^{\sigma}({\pmb k})\times \ldots \times \theta_s^{\sigma}({\pmb k})  : \ell).
\end{align*}
By Lemma \ref{Lemprodzeta}, this becomes
\begin{align*}
{\mathcal O}({\pmb k}: \ell)&=\sum_{\sigma'\in S_E^{\delta^{\dagger}}}\varepsilon_{\sigma}{\mathcal O}(\theta_s^{\sigma'}({\pmb k}^{\dagger})\times \theta_{s-1}^{\sigma'}({\pmb k}^{\dagger})\times \ldots \times \theta_1^{\sigma'}({\pmb k}^{\dagger})  : \ell)\\
&=\sum_{\sigma'\in S_E^{\delta^{\dagger}}}\varepsilon_{\sigma'}\sum_{\substack{{\pmb \varepsilon}'\in T_{\delta^{\dagger}}({\mathbb Z}_{\geq 0})\\
|{\pmb \varepsilon}'|=\ell}}\prod_{i=1}^s \zeta(\theta_i^{\sigma'}({\pmb k}^{\dagger}+{\pmb \varepsilon}'))\\
&=\sum_{\substack{{\pmb \varepsilon}'\in T_{\delta^{\dagger}}({\mathbb Z}_{\geq 0})\\
|{\pmb \varepsilon}'|=\ell}}\sum_{\sigma'\in S_E^{\delta^{\dagger}}}\varepsilon_{\sigma'}\prod_{i=1}^s \zeta(\theta_i^{\sigma'}({\pmb k}^{\dagger}+{\pmb \varepsilon}'))
={\mathcal O}({\pmb k}^{\dagger}: \ell).
\end{align*}
\hspace{15cm}$\Box$
%
%
%
%

%
\bigskip
\noindent
\textsc{Maki Nakasuji}\\
Department of Information and Communication Science, Faculty of Science, \\
 Sophia University, \\
 7-1 Kio-cho, Chiyoda-ku, Tokyo, 102-8554, Japan \\
 \texttt{nakasuji@sophia.ac.jp}

\medskip
\noindent
\textsc{Yasuo Ohno}\\
Mathematical Institute, \\
Tohoku University, \\
Sendai 980-8578, Japan \\
 \texttt{ohno.y@m.tohoku.ac.jp}

\end{document}